\documentclass[11pt]{amsart}

\usepackage[T1]{fontenc}
\usepackage{amsmath,amssymb,amsfonts,amsthm}
\usepackage{graphicx}
\usepackage{geometry}
\usepackage{enumerate}          
\usepackage{xcolor}             
\usepackage{colortbl}
\usepackage{float}
\usepackage{framed}
\usepackage{multicol}
\usepackage{tabularx}

\usepackage{pgf,tikz,pgfplots}
\usepackage{tkz-base}
\usepackage{tkz-fct}
\tikzset{>=Triangle}
\pgfplotsset{compat=1.17}

\usepackage{titlesec}
\titleformat{\section}
{\normalfont\large\bfseries}{\thesection.}{1em}{}
\titleformat{\subsection}
{\normalfont\normalsize\bfseries}{\thesubsection.}{1em}{}

\usepackage[pdftex,plainpages=false,pdfpagelabels,
unicode=true,
bookmarksopen=true,
colorlinks=true,
linkcolor=blue,
citecolor=blue,
urlcolor=blue]{hyperref}

\numberwithin{equation}{section}
\newtheorem{theorem}{Theorem}[section]
\newtheorem{proposition}[theorem]{Proposition}

\theoremstyle{definition}

\theoremstyle{remark}

\newcommand{\linf}{\ell^\infty}
\newcommand{\cseq}{c}

\begin{document}
	
	\title[Isometric embeddings of separable Banach spaces]%
	{Isometric embeddings of separable Banach spaces into $(\ell^\infty \setminus c)\cup\{0\}$}
	
	\author[Geivison Ribeiro]{Geivison Ribeiro}
\address{Department of Mathematics, IMECC, University of Campinas (Unicamp), 13083-970, Campinas–SP, Brazil}
\email{geivison.ribeiro@academico.ufpb.br}
\email{geivison@unicamp.br}

	
	\maketitle

\begin{abstract}
	The classical Banach--Mazur theorem asserts that every separable Banach space admits an isometric embedding into $C[0,1]$. 
	It is also well known that every separable Banach space embeds isometrically into $\ell^\infty$.
	We show that such an embedding can be chosen so that its image intersects $c$ only at the origin. Moreover, we prove that any finite- or countable-dimensional, or more generally separable, subspace of $(\ell^\infty \setminus c)\cup\{0\}$ can be extended to a subspace containing an isometric copy of an arbitrary separable Banach space, while still avoiding $c$. 
	We further establish that this extension property also holds for every subspace $D\subset \ell^\infty$ with $D\cap c=\{0\}$ and separable image in the quotient $\ell^\infty/c$.
\end{abstract}

\section{Introduction}

One of the classical landmarks of Banach space theory is the Banach--Mazur theorem, which asserts that every separable Banach space is isometric to a subspace of $C[0,1]$, (see, e.g., \cite{BessagaPelczynski1975}, \cite{KleiberPervin1969}).  
This theorem tells us that, despite the apparent variety of separable Banach spaces, they can all be realized within a single, very concrete function space.  
At the same time, it emphasizes the richness of $C([0,1])$, which is large enough to accommodate all of them.  

Even more generally, it was already noted by Banach \cite{Banach1932} that every separable Banach space can be embedded isometrically into $\ell^\infty$. 
Thus $\ell^\infty$ provides a natural environment in which all Banach spaces can be represented.  
Within this environment, certain subspaces have played a central role.  
The spaces $c$ of convergent sequences and $c_0$ of sequences converging to zero are classical examples whose presence reflects structural features that are impossible to overlook.  
It is in this setting that Lindenstrauss \cite{Lindenstrauss1967} and Rosenthal \cite{Rosenthal1968} established that every separable subspace of $\ell^\infty$ is quasi-complemented, 
while Wilansky \cite{Wilansky1975} and Kalton \cite{Kalton1975} developed fundamental criteria for understanding large closed subspaces of sequence spaces.  
More broadly, Kleiber and Pervin \cite{KleiberPervin1969} later showed that metric spaces of arbitrary density can be embedded in suitable $C(K)$ spaces, extending the reach of these ideas beyond the separable setting.  

Alongside this line of work, several remarkable refinements of the Banach--Mazur theorem have brought new depth to the picture.  
Rodr\'iguez--Piazza \cite{Piazza} proved that the isometric embedding of a separable Banach space into $C([0,1])$ can be chosen so that every nonzero function in the image is nowhere differentiable.  
Hencl \cite{Hencl2000} went further, showing that one can arrange for the image to consist of functions without approximate differentiability and even without H\"older continuity.  
Earlier, Fonf, Gurariy and Kadets \cite{FonfGurariyKadets1999} had already exhibited infinite-dimensional subspaces of $C([0,1])$ made entirely of nowhere differentiable functions.  
These contributions highlight a subtle phenomenon: even when universality is guaranteed, one can impose surprisingly strong restrictions on the nature of the embedding.  

\medskip

\noindent\textbf{Our contributions.}
In this article we investigate analogous phenomena in the sequence space $\ell^\infty$.  
Our main results can be summarized as follows:
\begin{itemize}
	\item Every separable Banach space admits an isometric embedding into $\ell^\infty$ whose image intersects
	$c$ only at the origin (\autoref{thm:main}).
	\item If $D \subset \ell^\infty$ is finite-dimensional with $D \cap c = \{0\}$, then $D$ can be extended to a closed
	subspace containing an isometric copy of any separable Banach space while still avoiding $c$
	(\autoref{prop:extension}).
	\item More generally, if $D$ has a countable Hamel basis or if $D$ is separable, then analogous
	extensions are possible, although the sum $D+T(E)$ need not be closed (autoref{prop:countable-dim}
	and \autoref{prop:separable-D}).
	\item In full generality, if $D\subset\ell^\infty$ satisfies $D\cap c=\{0\}$ and has separable image in the quotient
	$\ell^\infty/c$, then the same extension property holds (\autoref{thm:general-D}).
\end{itemize}

These results bring a new perspective on embeddings into $\ell^\infty$: not only do such embeddings exist in abundance, but they can also be arranged so as to avoid $c$ entirely, even when extensions are forced to contain prescribed subspaces.  
In this way, the work presented here complements both the classical geometric results 
\cite{Kalton1975,Lindenstrauss1967,Rosenthal1968,Wilansky1975}
 and more recent developments in lineability and spaceability 
\cite{ABRR-2025,AGPS,AGS,bernaljfa,FPT,Raposo,Leo,Papathanasiou2022,PR}

\section{Main results}\label{sec:main}

In this section we present the two main contributions of this work.

\medskip 

\begin{theorem}\label{thm:main}
	Every separable Banach space $E$ is isometric to a subspace $Y \subset \ell^\infty$ such that 
	\[
	Y \cap c = \{0\}.
	\]
\end{theorem}

\medskip

The next result shows that the above construction can be made more flexible by incorporating additional finite-dimensional subspaces.  

\medskip 

\begin{theorem}[Finite-dimensional extension]\label{prop:extension}
	Let $D \subset \ell^\infty$ be a finite-dimensional subspace with $D \cap c = \{0\}$. 
	For any separable Banach space $E$, there exists a closed subspace $Z \subset \ell^\infty$ containing $D$ and a subspace isometric to $E$, such that $Z \cap c = \{0\}$. 
	\bigskip
	
	In other words, any finite-dimensional subspace of $(\ell^\infty \setminus c) \cup \lbrace 0 \rbrace$ can be enlarged to a closed subspace that also contains an isometric copy of an arbitrary separable Banach space, while still intersecting $c$ only at the origin.
\end{theorem}

\medskip 

\noindent\textbf{Remark.}  
\autoref{thm:main} is a particular case of \autoref{prop:extension}, obtained by choosing $D = \{0\}$. 
For clarity of exposition, however, we prefer to present them in this order: first the embedding result, and then its refinement through finite-dimensional extensions.

\medskip

\begin{proof}[Proof of \autoref{thm:main}]  
	Let $E$ be separable and fix a dense sequence $(u_k)_{k=1}^{\infty}\subset \mathbb{S}_E$.  
	By Hahn--Banach, for each $k$ there exists $\phi_k\in\mathbb{S}_{E^*}$ with $\phi_k(u_k)=1$.
	
	Define $T:E\to \ell^\infty$ by  
	\[
	T(x):=\big(\psi_n(x)\big)_{n=1}^\infty,
	\qquad 
	\psi_{2k-1}:=\phi_k,\ \ \psi_{2k}:=-\phi_k.
	\]

	For every $x\in E$, 
	
	\[ \|T(x)\|_\infty = \sup_{n} |\psi_n(x)| \le \|x\|. \] 
	Conversely, let $x\neq 0$ and set $v:=x/\|x\|$. Since $(u_k)$ is dense in $\mathbb{S}_E$, pick a subsequence $(u_{k_j})_{j=1}^{\infty}$ with $u_{k_j}\to v$.  
	As $\|\phi_{k_j}\|=1$,
	\[
	|\phi_{k_j}(v)-\phi_{k_j}(u_{k_j})|\le \|v-u_{k_j}\|\longrightarrow 0.
	\]
	Since $\phi_{k_j}(u_{k_j})=1$, we get $\phi_{k_j}(v)\to 1$, hence $\phi_{k_j}(x)=\|x\|\,\phi_{k_j}(v)\to \|x\|$.  
	Therefore 
	\[ 
	\|T(x)\|_\infty\ge \limsup_{j}|\phi_{k_j}(x)|=\|x\|,
	\] and thus 
	
	\[ \|T(x)\|_\infty=\|x\|.\]

	For $x\neq 0$, from the previous step $\phi_{k_j}(x)\to \|x\|$, and by construction 
	\[
	\psi_{2k_j}(x)=-\phi_{k_j}(x)\to -\|x\|.
	\]  
	Hence $T(x)$ has subsequences converging to $\|x\|$ and to $-\|x\|$, so it does not converge.  
	Therefore $T(x)\notin c$ for $x\neq 0$, i.e., $T(E)\cap c=\{0\}$.
\end{proof}

	\medskip

	\begin{proof}[Proof of \autoref{prop:extension}]
		Fix a basis $w^1,\dots,w^r$ of the finite-dimensional subspace $D\subset \ell_\infty$ and set
		\[
		z_n := (w^1_n,\dots,w^r_n)\in \mathbb{K}^r \qquad (n\in\mathbb{N}),
		\]
		where $w^i=(w_n^i)_{n=1}^{\infty}$.
		Since each $w^i$ is bounded, $(z_n)_{n=1}^{\infty}$ is bounded in $\mathbb{K}^r$. By Bolzano--Weierstrass, there exists a subsequence $(z_{n_j})_{j=1}^{\infty}$ converging to some $\alpha=(\alpha_1,\dots,\alpha_r)\in \mathbb{K}^r$. Let
		\[
		I:=\{n_j:j\in\mathbb{N}\},\qquad I^+:=\{n_{2j}:j\in\mathbb{N}\},\qquad I^-:=\{n_{2j-1}:j\in\mathbb{N}\}.
		\]
		Then for every $d=\sum_{i=1}^r a_i w^i\in D$,
		\[
		d_n =\sum_{i=1}^r a_i w_n^i \xrightarrow[n\in I]{} \langle a,\alpha\rangle := \sum_{i=1}^r a_i\alpha_i,
		\]
		hence $(d_n)_{n\in I^+}\to \langle a,\alpha\rangle$ and $(d_n)_{n\in I^-}\to \langle a,\alpha\rangle$.
		
		\smallskip
		
		Let $(u_k)_{k=1}^{\infty}\subset S_E$ be dense (since $E$ is separable). For each $k\in\mathbb{N}$, by Hahn--Banach pick $\phi_k\in S_{E^*}$ with $\phi_k(u_k)=1$.
		Fix bijections $\eta^+:\mathbb{N}\to I^+$ and $\eta^-:\mathbb{N}\to I^-$. Define functionals $(\varphi_n)_{n=1}^{\infty}\subset E^*$ by
		\[
		\varphi_n :=
		\begin{cases}
			\ \ \ \phi_{\ (\eta^+)^{-1}(n)}, & n\in I^+,\\[2pt]
			-\phi_{\ (\eta^-)^{-1}(n)}, & n\in I^-,\\[2pt]
			\ 0, & n\notin I.
		\end{cases}
		\]
		Set $T:E\to\ell_\infty$ by $T(x):=(\varphi_n(x))_{n=1}^{\infty}$. Then $T$ is linear and $\|T(x)\|_\infty\le \|x\|$ for all $x$, since $\|\varphi_n\|\le 1$.
		
		\emph{Step 1. $T$ is an isometry.}
		Let $x\neq 0$ and put $v:=x/\|x\|\in S_E$. Choose a subsequence $(u_{k_j})_{j=1}^{\infty}$ with $u_{k_j}\to v$.
		As $\|\phi_{k_j}\|=1$, we have
		\[
		|\phi_{k_j}(v)-\phi_{k_j}(u_{k_j})|
		\le \|v-u_{k_j}\|\longrightarrow 0.
		\]
		Since $\phi_{k_j}(u_{k_j})=1$, it follows that $\phi_{k_j}(v)\to 1$, hence
		\[
		\phi_{k_j}(x)=\|x\|\phi_{k_j}(v)\longrightarrow \|x\|.
		\]
		Let $(n_j)_{j=1}^{\infty}$ with $n_j:=\eta^+(k_j)\in I^+$. Then $\varphi_{n_j}=\phi_{k_j}$, so (passing to an increasing subsequence if needed) $\varphi_{n_j}(x)\to\|x\|$, which yields
		\[
		\|T(x)\|_\infty \ \ge\ \limsup_{j\to\infty} |\varphi_{n_j}(x)|\ =\ \|x\|.
		\]
		Together with the reverse inequality, we conclude $\|T(x)\|_\infty=\|x\|$ for all $x$. Thus $T$ is an isometric embedding.
		
		\smallskip
		
		\emph{Avoiding $D+c$.}
		We show $T(E)\cap(D+c)=\{0\}$. Suppose $x\neq 0$ and $T(x)=d+y$ with $d\in D$ and $y\in c$.
		Write $d=\sum_{i=1}^r a_i w^i$ and let $L:=\langle a,\alpha\rangle$. As above, $(d_{n_j})_{j=1}^{\infty}\to L$ along both $I^+$ and $I^-$. On the other hand, defining $(p_j)_{j=1}^{\infty}$ with $p_j:=\eta^-(k_j)\in I^-$ we have $\varphi_{p_j}=-\phi_{k_j}$, hence
		\[
		\varphi_{n_j}(x)-d_{n_j}\ \longrightarrow\ \|x\|-L,\qquad 
		\varphi_{p_j}(x)-d_{p_j}\ \longrightarrow\ -\|x\|-L.
		\]
		Since $\|x\|>0$, these limits are distinct; thus $(T(x)-d)_n$ has two different subsequential limits (along $I^+$ and $I^-$), so it cannot converge. Therefore $y=T(x)-d\notin c$, a contradiction. Hence $T(x)\notin D+c$ for all $x\neq 0$, proving $T(E)\cap(D+c)=\{0\}$.
		
		\smallskip
		
		\emph{Conclusion.}
		Let $F:=T(E)$ and $G:=D+F$. Since $T$ is an isometry and $E$ is complete, $F$ is closed in $\ell_\infty$. The sum of a finite-dimensional subspace with a closed subspace is closed, hence $G$ is closed. If $z\in G\cap c$, write $z=d+T(x)$ with $d\in D$, $x\in E$. Then $T(x)=z-d\in c+D$, so $T(x)\in T(E)\cap(D+c)=\{0\}$, which gives $x=0$ and $z=d\in D\cap c=\{0\}$. Therefore $G\cap c=\{0\}$, $D\subset G$, and $G$ contains the isometric copy $F$ of $E$, as required.
	\end{proof}
	
			\medskip
		The arguments given so far rely on the existence of a subsequence $(n_j)$ along which 
		every element of $D$ has a limit. In the finite-dimensional or separable case, 
		this was ensured by a diagonal argument applied to a countable dense set. 
		However, for a general subspace $D\subset \ell^\infty$ with $D\cap c=\{0\}$, 
		such a construction cannot be carried out unless one imposes additional hypotheses. 
		A natural and sufficient condition is that $D$ be \emph{separable modulo $c$}, 
		that is, that the image $q(D)$ of $D$ under the quotient map $q:\ell^\infty\to\ell^\infty/c$ 
		be separable. This assumption exactly captures what is needed in Step~1 of the proof, 
		namely, the existence of a countable dense family of representatives in $D$ 
		modulo $c$ to which the diagonal procedure can be applied.
		\medskip
		
		We now state the general extension result under this hypothesis.
		
		\begin{theorem}[Extension with $D$ transversal to $c$ and separable modulo $c$]
			Let $D\subset \ell^\infty$ satisfy $D\cap c=\{0\}$, and let $q:\ell^\infty\to \ell^\infty/c$ be the quotient map.
			Assume $q(D)$ is separable in $\ell^\infty/c$ (i.e., $D$ is separable modulo $c$).
			Then for every separable Banach space $E$ there exists an isometry $T:E\to\ell^\infty$ such that
			\[
			T(E)\cap (D+c)=\{0\}.
			\]
			In particular, for $Z:=D+T(E)$ we have $Z\cap c=\{0\}$. 
		\end{theorem}
		
		\begin{proof}
			
			Since $q(D)$ is separable, choose a dense sequence $(y^j)_{j=1}^\infty\subset q(D)$.
			Pick representatives $d^j\in D$ with $q(d^j)=y^j$.
			By a diagonal argument, there is a strictly increasing sequence $(n_k)_{k=1}^\infty$ such that
			$(d^j_{\,n_k})_k$ converges in $\Bbb K$ for every $j$.
			Define $I:=\{n_k:k\in\Bbb N\}$ and split $I$ into two infinite subsequences
			\[
			I^+:=\{n_{2k}:k\in\Bbb N\},\qquad I^-:=\{n_{2k-1}:k\in\Bbb N\}.
			\]
			Set $L(d^j):=\lim_{k\to\infty} d^j_{\,n_k}$.

			For $d\in D$ and $\eta>0$, choose $j$ with $\|q(d)-q(d^j)\|_{\ell^\infty/c}<\eta/3$. 
			Then there exists $u\in c$ such that
			\[
			\|(d-d^j)-u\|_\infty<\eta/3,
			\]
			that is,
			\[
			|d_n-d^j_n-u_n|<\eta/3 \qquad \text{for all }n\in\mathbb N.
			\]
			Fix $p,q\in\mathbb N$. We insert and subtract $d^j$ and $u$ at both ends:
			\[
			\begin{aligned}
				d_{n_p}-d_{n_q}
				&= (d_{n_p}-d^j_{n_p}-u_{n_p}) + (d^j_{n_p}-d^j_{n_q}) + (u_{n_p}-u_{n_q}) + (d^j_{n_q}+u_{n_q}-d_{n_q}).
			\end{aligned}
			\]
			Taking absolute values and using the triangle inequality gives
			\[
			\begin{aligned}
				|d_{n_p}-d_{n_q}|
				&\le |d_{n_p}-d^j_{n_p}-u_{n_p}| 
				+ |d^j_{n_p}-d^j_{n_q}|
				+ |u_{n_p}-u_{n_q}|
				+ |d^j_{n_q}+u_{n_q}-d_{n_q}|.
			\end{aligned}
			\]
			By construction, the first and last terms are both $<\eta/3$. 
			Since $(d^j_{n_k})_k$ converges, there exists $K_1$ such that $|d^j_{n_p}-d^j_{n_q}|<\eta/3$ whenever $p,q\ge K_1$. 
			Since $u\in c$, the sequence $(u_{n_k})_k$ converges, so there exists $K_2$ such that $|u_{n_p}-u_{n_q}|<\eta/3$ whenever $p,q\ge K_2$. 
			Therefore, for $p,q\ge \max\{K_1,K_2\}$,
			\[
			|d_{n_p}-d_{n_q}| < \eta/3 + \eta/3 + \eta/3 + \eta/3 = \tfrac{4}{3}\eta.
			\]
			As $\eta>0$ was arbitrary, this shows that $(d_{n_k})_k$ is Cauchy, hence convergent. 
			\[
			L(d):=\lim_{k\to\infty} d_{\,n_k}.
			\]
			Moreover, passing to the limit in $d_{\,n_k}=d^j_{\,n_k}+u_{\,n_k}+r_k$ with $|r_k|<\eta/3$ gives
			\[
			L(d)=L(d^j)+\ell(u),
			\]
			so $L:D\to\Bbb K$ is well defined and (by the argument above) independent of the choice of $j$ and $u$.
			As usual, the limit along $I^+$ and along $I^-$ also equals $L(d)$.
			
			\smallskip
			\emph{Step 2. Construction of $T$.}
			Partition $I$ into disjoint infinite blocks
			\[
			I=\bigsqcup_{j=1}^\infty (I_j^+\sqcup I_j^-).
			\]
			Let $(f_j)_{j=1}^\infty\subset S_{E^*}$ be dense. Define $T:E\to\ell^\infty$ by
			\[
			(Tx)_n :=
			\begin{cases}
				f_j(x), & n\in I_j^+,\\
				-\,f_j(x), & n\in I_j^-,\\
				0, & n\notin I.
			\end{cases}
			\]
			Then $\|Tx\|_\infty=\sup_j |f_j(x)|=\|x\|$, hence $T$ is an isometry.
			
			\smallskip
			\emph{Step 3. Separation from $D+c$.}
			If $Tx=d+u$ with $d\in D$, $u\in c$, then along $I_j^\pm$ we have
			\[
			\lim_{n\in I_j^+}(Tx)_n=f_j(x),\qquad \lim_{n\in I_j^-}(Tx)_n=-\,f_j(x),
			\]
			while, by Step~1,
			\[
			\lim_{n\in I_j^\pm} d_n=L(d),\qquad \lim_{n\in I_j^\pm} u_n=\ell(u).
			\]
			Hence $f_j(x)=L(d)+\ell(u)=-f_j(x)$ for all $j$, so $f_j(x)=0$ for all $j$, and therefore $x=0$.
			Thus $T(E)\cap(D+c)=\{0\}$.
			
			\smallskip
			\emph{Step 4. Conclusion.}
			Let $Z:=D+T(E)$. If $z\in Z\cap c$, write $z=d+Tx$. By Step~3, $x=0$, hence $z=d\in D\cap c=\{0\}$.
			Therefore $Z\cap c=\{0\}$.
		\end{proof}
			\medskip

		 \medskip
		 
		\noindent\textbf{Open question.} 
		In \autoref{prop:extension} we showed that, given a finite-dimensional subspace $D\subset (\linf\setminus \cseq)\cup \lbrace 0 \rbrace$ and a separable Banach space $E$, one can construct an isometry $T:E\to \linf$ and a separable closed subspace $Z\subset \linf$ containing $D$ such that
		\[
		T(E)\subset Z \quad \text{and} \quad Z\cap \cseq=\{0\}.
		\]
		
		A natural question is whether one can go further and obtain \emph{nonseparable} extensions:
		
		\begin{quote}
			Given $D\subset (\linf\setminus \cseq) \cup \lbrace 0 \rbrace$ finite-dimensional and $E$ an arbitrary separable Banach space, does there always exist an isometry $T:E\to \linf$ and a closed nonseparable subspace  $Z\subset \linf$ containing $D$ such that
			\[
			T(E) \subset Z, \qquad Z\cap \cseq=\{0\}?
			\]
		\end{quote}
		
\section{Extensions beyond the finite-dimensional case}\label{sec:extensions-beyond}
\medskip 

In \autoref{prop:extension} we established that if $D\subset \ell^\infty$ is finite-dimensional with $D\cap c=\{0\}$, then 
for every separable Banach space $E$ there exists an isometric embedding $T:E\to \ell^\infty$ such that 
\[
T(E)\cap(D+c)=\{0\}.
\]
Moreover, in this situation the extension $Z:=D+T(E)$ is closed, since the sum of a closed subspace with a finite-dimensional subspace is always closed.
In the framework of $(\alpha,\beta)$-structures (see \cite{PR, Raposo}), this already yields that $\ell^\infty\setminus c$ is $(r,\mathfrak c)$-\emph{spaceable} for every $r\in\mathbb N$.

\medskip

\noindent\textbf{Reminder: $(\alpha,\beta)$-lineability/spaceability.}
Let $A\subset V$ where $V$ is a (topological) vector space, and let $\alpha\le \beta$ be cardinals.
We say that:
\begin{itemize}
	\item $A$ is \emph{$(\alpha,\beta)$-lineable} if $A \cup \lbrace 0 \rbrace$ contains an $\alpha$-dimensional subspace and, moreover, for every $\alpha$-dimensional subspace $W_\alpha\subset A\cup\{0\}$ there exists a $\beta$-dimensional subspace $W_\beta$ with
	\[
	W_\alpha\subset W_\beta\subset A\cup\{0\}.
	\]
	\item $A$ is \emph{$(\alpha,\beta)$-spaceable} if, in addition, $V$ is topological and $W_\beta$ can be chosen closed.
\end{itemize}

\medskip

A natural question is whether the same extension property persists for larger classes of subspaces $D$. 
We address here two intermediate cases: (i) when $D$ has a countable Hamel basis, and (ii) when $D$ is separable (but possibly without a countable Hamel basis). 
In both situations we construct the required isometry $T$ and retain the separation property $T(E)\cap(D+c)=\{0\}$.

\begin{proposition}[Extension with countable-dimensional $D$]\label{prop:countable-dim}
	Let $D\subset \ell^\infty$ be a subspace with a countable Hamel basis $W=\{w^i\}_{i=1}^\infty$ and assume $D\cap c=\{0\}$.
	Let $E$ be a separable Banach space. Then there exists an isometry $T:E\to \ell^\infty$ such that
	\[
	T(E)\cap (D+c)=\{0\}.
	\]
	In particular, there exists $Z)$ such thata $D\subset Z$, $Z$ contains an isometric copy of $E$, and $Z\cap c=\{0\}$. 
	(Closedness of $Z$ is not guaranteed in general.)
\end{proposition}

\begin{proof}
	
	For each $i\ge 1$ write $w^i=(w^i_n)_{n=1}^{\infty}\in \ell^\infty$. Each coordinate sequence $(w^i_n)_{n=1}^{\infty}$ is bounded.  
	By a diagonal argument, we construct a strictly increasing sequence $(n_j)_{j=1}^{\infty}$ such that
	\[
	\forall i\in\mathbb N:\quad \lim_{j\to\infty} w^i_{\,n_j}=: \alpha_i\in\Bbb K.
	\]
	Set $I:=\{n_j:j\in\mathbb N\}$ and split it into two infinite subsequences
	\[
	I^+:=\{n_{2j}:j\in\mathbb N\},\qquad I^-:=\{n_{2j-1}:j\in\mathbb N\}.
	\]
	Thus, for any $d=\sum_{i=1}^m a_i w^i\in D$, we have
	\[
	(d_{n_j})_{j=1}^{\infty}\ \longrightarrow\ \sum_{i=1}^m a_i \alpha_i=:L(d),
	\]
	so $d_{n_j}\to L(d)$ along both $I^+$ and $I^-$.  
	
	From here the argument is identical to the proof of Theorem~\ref{prop:extension}: we construct $T:E\to\ell^\infty$ using functionals $\pm\phi_k$ placed alternately on $I^\pm$, show that $T$ is an isometry, and use the limits $L(d)$ to deduce $T(E)\cap(D+c)=\{0\}$. The assertions about $Z:=D+T(E)$ follow immediately.
\end{proof}
\medskip

\noindent\emph{\textbf{$(\alpha,\beta)$-interpretation.}}
Since every countable-dimensional $D$ embeds as above, we obtain that $\ell^\infty\setminus c$ is $(\aleph_0,\mathfrak c)$-\emph{lineable}.

\begin{proposition}[Extension with separable $D$]\label{prop:separable-D}
	Let $D\subset \ell^\infty$ be a separable subspace with $D\cap c=\{0\}$, and let $E$ be a separable Banach space.
	Then there exists an isometry $T:E\to \ell^\infty$ such that
	\[
	T(E)\cap (D+c)=\{0\}.
	\]
	In particular, there exists $Z$ such that $D\subset Z$, $Z$ contains an isometric copy of $E$, and $Z\cap c=\{0\}$. 
	(Closedness of $Z$ is not guaranteed in general.)
\end{proposition}

\begin{proof}
	Let $(v^j)_{j=1}^{\infty}\subset D$ be dense in $\|\cdot\|_\infty$. 
	By a diagonal argument, there exists a strictly increasing sequence $(n_j)_{j=1}^{\infty}$ such that
	\[
	\forall\, j\in\mathbb N:\quad (v^j_{\,n_k})_{k=1}^{\infty}\ \text{converges in }\Bbb K.
	\]
	Set $I:=\{n_j:j\in\mathbb N\}$ and split it into two infinite subsets
	\[
	I^+:=\{n_{2j}:j\in\mathbb N\},\qquad I^-:=\{n_{2j-1}:j\in\mathbb N\}.
	\]
	
	\emph{Claim.} For every $d=(d_{\,n})_{n=1}^{\infty}\in D$, the sequence $(d_{\,n_j})_{j=1}^{\infty}$ converges (hence has the same limit along $I^+$ and $I^-$).
	
	\emph{Proof of the claim.}
	Fix $d=(d_{\,n})_{n=1}^{\infty}\in D$ and let $\eta>0$. By density of $(v^j)_{j=1}^{\infty}$ in $D$ for the sup norm, choose $m\in\mathbb N$ such that
	\[
	\|d-v^m\|_\infty<\eta/3.
	\]
	By the definition of the sup norm, this gives the \emph{uniform} estimate
	\[
	|d_{\,n_k}-v^m_{\,n_k}|\le \|d-v^m\|_\infty<\eta/3\qquad\text{for all }k\in\mathbb N.
	\]
	Since $(v^m_{\,n_k})_{k}$ converges, it is Cauchy; hence there exists $K\in\mathbb N$ such that, for all $p,q\ge K$,
	\[
	|v^m_{\,n_p}-v^m_{\,n_q}|<\eta/3.
	\]
	Then, for $p,q\ge K$, the triangle inequality yields
	\[
	\begin{aligned}
		|d_{\,n_p}-d_{\,n_q}|
		&\le |d_{\,n_p}-v^m_{\,n_p}| + |v^m_{\,n_p}-v^m_{\,n_q}| + |v^m_{\,n_q}-d_{\,n_q}|\\
		&< \eta/3+\eta/3+\eta/3=\eta.
	\end{aligned}
	\]
	Thus $(d_{\,n_j})_{j}$ is Cauchy and, since $\Bbb K$ is complete, it converges. 
	Denote its limit by
	\[
	L(d):=\lim_{j\to\infty} d_{\,n_j}.
	\]
	As a convergent sequence has all subsequences converging to the same limit, the limits of $(d_{\,n_j})$ along $I^+$ and along $I^-$ both equal $L(d)$. \qedhere
	
	\smallskip
	From this point on, the proof is identical to that of \autoref{prop:extension}: 
	we place the functionals $\pm f_j$ alternately on $I^\pm$ to define $T:E\to\ell^\infty$, 
	verify that $T$ is an isometry, and use the limits $L(d)$ to conclude that 
	$T(E)\cap(D+c)=\{0\}$. The statements about $Z:=D+T(E)$ follow immediately.
\end{proof}
\medskip

\noindent\emph{\textbf{$(\alpha,\beta)$-interpretation.}}
If $\alpha=\dim D$ (possibly infinite), the result shows that $\ell^\infty\setminus c$ is $(\alpha,\mathfrak c)$-\emph{lineable}.

\medskip

The arguments given so far rely on the existence of a subsequence $(n_j)$ along which 
every element of $D$ has a limit. In the finite-dimensional or separable case, 
this was ensured by a diagonal argument applied to a countable dense set. 
However, for a general subspace $D\subset \ell^\infty$ with $D\cap c=\{0\}$, 
such a construction cannot be carried out unless one imposes additional hypotheses. 
A natural and sufficient condition is that $D$ be \emph{separable modulo $c$}, 
that is, that the image $q(D)$ of $D$ under the quotient map $q:\ell^\infty\to\ell^\infty/c$ 
be separable. This assumption exactly captures what is needed in Step~1 of the proof, 
namely, the existence of a countable dense family of representatives in $D$ 
modulo $c$ to which the diagonal procedure can be applied.
\medskip

We now state the general extension result under this hypothesis.
\medskip

\begin{theorem}[Extension with $D$ transversal to $c$ and separable modulo $c$]\label{thm:general-D}
	Let $D\subset \ell^\infty$ satisfy $D\cap c=\{0\}$, and let $q:\ell^\infty\to \ell^\infty/c$ be the quotient map.
	Assume $q(D)$ is separable in $\ell^\infty/c$ (i.e., $D$ is separable modulo $c$).
	Then for every separable Banach space $E$ there exists an isometry $T:E\to\ell^\infty$ such that
	\[
	T(E)\cap (D+c)=\{0\}.
	\]
	In particular, for $Z:=D+T(E)$ we have $Z\cap c=\{0\}$. 
\end{theorem}

\begin{proof}
	
	Since $q(D)$ is separable, choose a dense sequence $(y^j)_{j=1}^\infty\subset q(D)$.
	Pick representatives $d^j\in D$ with $q(d^j)=y^j$.
	By a diagonal argument, there is a strictly increasing sequence $(n_k)_{k=1}^\infty$ such that
	$(d^j_{\,n_k})_k$ converges in $\Bbb K$ for every $j$.
	Define $I:=\{n_k:k\in\Bbb N\}$ and split $I$ into two infinite subsequences
	\[
	I^+:=\{n_{2k}:k\in\Bbb N\},\qquad I^-:=\{n_{2k-1}:k\in\Bbb N\}.
	\]
	Set $L(d^j):=\lim_{k\to\infty} d^j_{\,n_k}$.

	For $d\in D$ and $\eta>0$, choose $j$ with $\|q(d)-q(d^j)\|_{\ell^\infty/c}<\eta/3$. 
	Then there exists $u\in c$ such that
	\[
	\|(d-d^j)-u\|_\infty<\eta/3,
	\]
	that is,
	\[
	|d_n-d^j_n-u_n|<\eta/3 \qquad \text{for all }n\in\mathbb N.
	\]
	Fix $p,q\in\mathbb N$. We insert and subtract $d^j$ and $u$ at both ends:
	\[
	\begin{aligned}
		d_{n_p}-d_{n_q}
		&= (d_{n_p}-d^j_{n_p}-u_{n_p}) + (d^j_{n_p}-d^j_{n_q}) + (u_{n_p}-u_{n_q}) + (d^j_{n_q}+u_{n_q}-d_{n_q}).
	\end{aligned}
	\]
	Taking absolute values and using the triangle inequality gives
	\[
	\begin{aligned}
		|d_{n_p}-d_{n_q}|
		&\le |d_{n_p}-d^j_{n_p}-u_{n_p}| 
		+ |d^j_{n_p}-d^j_{n_q}|
		+ |u_{n_p}-u_{n_q}|
		+ |d^j_{n_q}+u_{n_q}-d_{n_q}|.
	\end{aligned}
	\]
	By construction, the first and last terms are both $<\eta/3$. 
	Since $(d^j_{n_k})_k$ converges, there exists $K_1$ such that $|d^j_{n_p}-d^j_{n_q}|<\eta/3$ whenever $p,q\ge K_1$. 
	Since $u\in c$, the sequence $(u_{n_k})_k$ converges, so there exists $K_2$ such that $|u_{n_p}-u_{n_q}|<\eta/3$ whenever $p,q\ge K_2$. 
	Therefore, for $p,q\ge \max\{K_1,K_2\}$,
	\[
	|d_{n_p}-d_{n_q}| < \eta/3 + \eta/3 + \eta/3 + \eta/3 = \tfrac{4}{3}\eta.
	\]
	As $\eta>0$ was arbitrary, this shows that $(d_{n_k})_k$ is Cauchy, hence convergent. 
	\[
	L(d):=\lim_{k\to\infty} d_{\,n_k}.
	\]
	Moreover, passing to the limit in $d_{\,n_k}=d^j_{\,n_k}+u_{\,n_k}+r_k$ with $|r_k|<\eta/3$ gives
	\[
	L(d)=L(d^j)+\ell(u),
	\]
	so $L:D\to\Bbb K$ is well defined and (by the argument above) independent of the choice of $j$ and $u$.
	As usual, the limit along $I^+$ and along $I^-$ also equals $L(d)$.
	
	\smallskip
	
	Partition $I$ into disjoint infinite blocks
	\[
	I=\bigsqcup_{j=1}^\infty (I_j^+\sqcup I_j^-).
	\]
	Let $(f_j)_{j=1}^\infty\subset S_{E^*}$ be dense. Define $T:E\to\ell^\infty$ by
	\[
	(Tx)_n :=
	\begin{cases}
		f_j(x), & n\in I_j^+,\\
		-\,f_j(x), & n\in I_j^-,\\
		0, & n\notin I.
	\end{cases}
	\]
	Then $\|Tx\|_\infty=\sup_j |f_j(x)|=\|x\|$, hence $T$ is an isometry.

	If $Tx=d+u$ with $d\in D$, $u\in c$, then along $I_j^\pm$ we have
	\[
	\lim_{n\in I_j^+}(Tx)_n=f_j(x),\qquad \lim_{n\in I_j^-}(Tx)_n=-\,f_j(x),
	\]
	while, by Step~1,
	\[
	\lim_{n\in I_j^\pm} d_n=L(d),\qquad \lim_{n\in I_j^\pm} u_n=\ell(u).
	\]
	Hence $f_j(x)=L(d)+\ell(u)=-f_j(x)$ for all $j$, so $f_j(x)=0$ for all $j$, and therefore $x=0$.
	Thus $T(E)\cap(D+c)=\{0\}$.
	
\smallskip
	\emph{Conclusion.}
	Let $Z:=D+T(E)$. If $z\in Z\cap c$, write $z=d+Tx$. By Step~3, $x=0$, hence $z=d\in D\cap c=\{0\}$.
	Therefore $Z\cap c=\{0\}$.
\end{proof}

\medskip

\noindent\textbf{Open question.}
Let $D\subset \ell^\infty$ be an arbitrary (possibly nonseparable) subspace with $D\cap c=\{0\}$, and let $E$ be a Banach separable.
Does there always exist an isometry $T:E\to \ell^\infty$ such that $T(E)\cap(D+c)=\{0\}$?

	\section*{Acknowledgments}
	The author is grateful to F. Costa Jr.\ and A.\ Raposo Jr.\ for carefully reading parts of the manuscript and for helpful comments on some proofs. 
	The author was supported by CAPES --- Coordena\c{c}\~ao de Aperfei\c{c}oamento de Pessoal de N\'ivel Superior (Brazil) --- through a postdoctoral fellowship at IMECC, Universidade Estadual de Campinas (PIPD/CAPES).

	\bigskip
	\noindent\textbf{2020 MSC:} 46B04, 46B20.  
	\textbf{Keywords:} isometric embeddings; separable Banach spaces; $\ell^\infty$; convergent sequences; lineability; spaceability.

\end{document}